\documentclass[11pt]{amsart}
\usepackage{amsfonts}
\usepackage{amssymb,latexsym}
\usepackage{enumerate}
\usepackage{mathrsfs}
\usepackage{amsmath}
\usepackage{amsthm}
\usepackage{hyperref}
\usepackage[square, comma, sort&compress, numbers]{natbib}
\usepackage{bm}
\usepackage{fancyhdr}
\usepackage{lastpage}
\usepackage{layout}
\newtheorem{thm}{Theorem}[section]
\newtheorem{prop}[thm]{Proposition}

\newtheorem{lemma}[thm]{Lemma}

\numberwithin{equation}{section}

\newcommand{\pt}{\partial}
\newcommand{\bp}{\overline{\partial}}
\newcommand{\lmd}{\lambda}
\newcommand{\vps}{\varepsilon}
\newcommand{\vph}{\varphi}
\newcommand{\og}{\omega}
\newcommand{\pal}{\parallel}
\newcommand{\raw}{\rightarrow}
\newcommand{\lraw}{\longrightarrow}
\newcommand{\mf}{\mathcal{F}}
\newcommand{\tr}{\text{tr}}
\newcommand{\therm}{\text{Herm}}
\newcommand{\tend}{\text{End}}
\newcommand{\lag}{\langle}
\newcommand{\rag}{\rangle}
\newcommand{\mfd}{\mathcal{D}}
\newcommand{\mfk}{\mathcal{K}}
\newcommand{\ol}{\overline}

\begin{document}

\title{\textsc{Semistable Higgs bundles  over compact Gauduchon manifolds }}
\author{Yanci Nie and Xi Zhang}
\address{Yanci Nie\\School of Mathematical Sciences\\
University of Science and Technology of China\\
Hefei, 230026\\ } \email{nieyanci@mail.ustc.edu.cn}
\address{Xi Zhang\\Key Laboratory of Wu Wen-Tsun Mathematics\\ Chinese Academy of Sciences\\School of Mathematical Sciences\\
University of Science and Technology of China\\
Hefei, 230026,P.R. China\\ } \email{mathzx@ustc.edu.cn}
\thanks{The authors were supported in part by NSF in China,  No.11131007.}
\maketitle
\begin{abstract}
In this paper, we consider the existence of approximate Hermitian-Einstein structure and the semi-stability on Higgs bundles over compact Gauduchon manifolds. By using the continuity method, we show that they are equivalent.
\end{abstract}

\section{introduction}

Let $X$ be an $n$-dimensional compact complex manifold and $g$ be a Hermitian metric with associated  K\"ahler form $\og$. $g$ is called to be Gauduchon if $\og$ satisfies $\pt\bp \og^{n-1}=0.$  It has been proved by Gauduchon that if $X$ is compact, there exists a Gauduchon metric (\cite{Gaud}) in the conformal class of every Hermitian metric $g$. In the following, we assume $\og$ is  Gauduchon.

Let $(L,h)$ be a Hermitian line bundle over $X$.  The $\og$-degree of $L$ is defined by
\begin{equation*}
\deg_{\og}(L):=\int_X c_1(L, A_h)\wedge\displaystyle{\frac{\og^{n-1}}{(n-1)!}},
\end{equation*}
where $c_1(L,A_h)$ is the first Chern form of $L$ associated with the induced Chern connection  $A_h$.  Since $\pt\bp \og^{n-1}=0,$  $\deg_{\og}(L)$ is  well defined and independent of the choice of    metric $h$ (\cite[p.~34-35]{MA}). Now given a rank $s$ coherent analytic sheaf  $\mathcal{F}$, we consider the  determinant line bundle  $\det{\mathcal{F}}=(\wedge^s\mathcal{F} )^{**}.$  Define the $\og$-degree of $\mathcal{F}$ by
\begin{equation*}
\deg_{\og}(\mathcal{F}):=\mbox{deg}_{\og}(\det{\mathcal{F}}).
\end{equation*}
If $\mf$ is non-trivial and torsion free, the $\og$-slope of $\mathcal{F}$ is defined by
  $$\mu_{\og}(\mathcal{F})=\frac{\mbox{deg}_{\og}(\mathcal{F})}{\mbox{rank}(\mathcal{F})}.$$

 Let $(E, \bp_E)$ be a holomorphic vector bundle over $X$. We say $E$ is $\og$-stable ($\og$-semi-stable) in the sense of Mumford-Takemoto  if for every proper coherent sub-sheaf $ \mathcal{F}\hookrightarrow E$, there holds
\begin{equation*}
\mu_{\og}(\mathcal{F})<\mu_{\og}(E)(\mu_{\og}(\mathcal{F})\leq\mu_{\og}(E)).
\end{equation*}

A Hermitian metric $H$ on $E$ is said to be $\og$-Hermitian-Einstein if the Chern curvature $F_H$  satisfies the Einstein condition
 $$ \sqrt{-1}\Lambda_{\og}F_H=\lmd\cdot \mathrm{Id}_{E},$$
 where $\lmd=\displaystyle{\frac{2\pi \mu_{\og}(E)}{Vol(X)}}.$  When the K\"ahler form is understood, we omit the subscript  $\og$ in the above definitions.

\medskip
 The Donaldson-Uhlernbeck-Yau theorem states that holomorphic vector bundles admit  Hermitian-Einstein metrics if they are stable.  It was proved  by  Narasimhan and Seshadri  in \cite{NarSe} for compact Riemann surface case, by Donaldson in \cite{Don2,Don3} for algebraic manifolds and by Uhlenbeck and Yau in \cite{UhYau, UhYau1} for general compact K\"ahler manifolds. The inverse problem that a holomorphic bundle admitting such a metric must be poly-stable( i.e. a direct sum of stable bundles with the same slope) was solved by Kobayashi \cite{Kobayashi1} and L\"ubke \cite{Lubke} independently. Actually, this is the well-known  Hitchin-Kobayashi correspondence for holomorphic vector bundles over compact K\"ahler manifolds.  This correspondence is also valid for compact Gauduchon manifolds \cite{bu,LY,MA}. There are many other  interesting generalized Hitchin-Kobayashi correspondences (see the references \cite{alvarez,bando,biquard,bradlow,bragar,garcia,Hit,jost,ljy1,ljy2,simpson1988constructing} for details).

\medskip
 A Higgs vector bundle $(E, \bp_E, \phi)$ over $X$ is a holomorphic vector bundle $(E,\bp_{E})$ together with a Higgs field $\phi\in \Omega_X^{1,0}(\text{End}(E))$ satisfying $\bp_E \phi=0$ and $\phi\wedge \phi=0.$  Higgs bundle was introduced by Hitchin \cite{Hit} in his study of self dual equations on a Riemann surface, and studied by Simpson \cite{simpson1988constructing} in his work on nonabelian Hodge theory. It has a rich structure and plays an important role in many areas including gauge theory, K\"ahler geometry and hyperk\"ahler geometry, group representations and non-abelian Hodge theory.  A Higgs bundle $(E, \overline{\partial }_{E}, \phi)$ is stable(resp. semi-stable) if $\mu(\mathcal{F})<\mu(E)(\mbox{resp}.\  \mu(\mathcal{F})
\leq\mu(E))$ for every proper  $\phi$-invariant coherent subsheaf $\mathcal{F}$ of $E$.

 Given a Hermitian metric $H$ on a Higgs bundle, we consider the Hitchin-Simpson connection  (\cite{simpson1988constructing})
 \begin{eqnarray*}
 D_{H, \overline{\partial }_{E}, \phi }= D_{H, \overline{\partial }_{E}} + \phi +\phi ^{\ast H},
 \end{eqnarray*}
where $D_{H, \bp_{E}}$ is the Chern connection, and $\phi ^{\ast H}$ is the adjoint of $\phi $ with respect to the metric $H$. The curvature of this connection is
\begin{eqnarray*}
F_{H, \overline{\partial }_{E}, \phi}=F_{H} +[\phi , \phi ^{\ast H}] +\pt_H \phi + \overline{\partial }_{E} \phi^{\ast H},
\end{eqnarray*}
where $F_{H}$ is the curvature of the Chern connection $D_{H, \overline{\partial }_{E}}$. A Hermitian metric $H$ on Higgs bundle $(E, \overline{\partial }_{E}, \phi  )$ is said to be
 Hermitian-Einstein  if the curvature $F_{H, \overline{\partial }_{E}, \phi}$  satisfies
\begin{eqnarray*}
\sqrt{-1}\Lambda_{\og}F_{H, \bp_E, \phi}=\sqrt{-1}\Lambda_{\og} (F_{H} +[\phi , \phi ^{\ast H}])
=\lambda Id_{E}.
\end{eqnarray*}

 Hitchin \cite{Hit} and Simpson \cite{simpson1988constructing} proved that a Higgs bundle is poly-stable if and only if it admits a Hermitian-Einstein structure. This is a Higgs bundle  version of the classical Hitchin-Kobayashi correspondence.
\par

\medskip

 A Higgs bundle is said to be admitting an approximate Hermitian-Einstein structure, if for
$\forall \vps>0,$ there exists a Hermitian metric $H_{\vps}$ such that
 $$\max\limits_X \mid \sqrt{-1}\Lambda_{\og}( F_{H_{\vps}}+[\phi,\phi^{*H_{\vps}}])-\lmd\cdot \mathrm{Id}_E \mid_{H_{\vps}}<\vps. $$

Kobayashi(\cite{Kobayashi}) introduced this notion in a holomorphic vector bundle (i.e. $\phi=0$). He proved that over a compact K\"ahler manifold, a holomorphic vector bundle  admitting such a structure structure must be semi-stable. In \cite{bruzzo},  Bruzzo and  Gra\~na Otero  generalized the above result to Higgs bundles. When $X$ is projective,  Kobayashi \cite{Kobayashi} solved the inverse part that a semi-stable holomorphic vector bundle must admit an approximate Hermitian-Einstein structure  and conjectured that this should be true for general K\"ahler manifolds. This was confirmed in \cite{Cardona,Jacob1,li2012existence}.
\medskip

In this paper, we are interested in the existence of approximate Hermitian-Einstein structures on  Higgs bundles over compact Gauduchon  manifolds. In fact, we prove that:
\begin{thm}\label{B}
Let $(X,\og)$ be an $n$-dimensional  compact Gauduchon manifold and $(E, \bp_E, \phi)$ be a rank $r$  Higgs bundle over $X$. Then  $(E, \bp_E, \phi)$ is semi-stable if and only if it admits an approximate Hermitian-Einstein structure.
\end{thm}

Now we give an overview of our proof. The difficult part of Theorem \ref{B} is to prove the existence of approximate Hermitian-Einstein structure. In the K\"ahler case, by using the Donaldson heat flow, Li and Zhang (\cite{li2012existence}) showed that the semi-stability implies admitting an approximate Hermitian-Einstein structure. Their proof relies on the properties of the Donaldson functional. However, the Donaldson functional is not well-defined if $\og$ is only Gauduchon. So Li and Zhang's argument can not be generalized to Gauduchon manifold case directly. In this paper,  we use the continuity  method  to prove the existence. Fixed a proper background  Hermitian metric $H_0$ on $E$, we  consider the following perturbed equation
 \begin{equation}\label{eq11}
L_{\varepsilon}(f):=\mathcal{K}_H -\lmd\mbox{Id}_{E}+\vps\log{f}=0,\ \ \ \ \ \vps\in (0,1],
 \end{equation}
  where $\mfk_H=\sqrt{-1}\Lambda_{\og}F_{H, \bp_E, \phi}=K_H+\sqrt{-1}\Lambda_{\og}[\phi,\phi^{*H}]$ and $f=H_0^{-1}\cdot H.$ It is obvious that $f$ and $\log f$ are self adjoint with respect to $H_0$ and $H$.  By the results  of L\"ubke and Teleman in \cite{MA,MA1}, (\ref{eq11}) is solvable for $\forall \vps\in (0,1]$.  Under the assumption of semi-stability, we can show that
  \begin{equation}
  \lim \limits_{\vps\rightarrow 0}\vps\max\limits_X \mid \log f_{\vps}\mid_{H_0}=0.
  \end{equation}
  This implies that  $\max\limits_X\mid \mathcal{K}_{H_{\vps}}-\lmd\cdot \mathrm{Id}_E \mid_{H_{\vps}}$ converges to zero as $\vps\rightarrow 0$ (see Theorem \ref{A.5} for details).

\medskip

This article is organised as below. In Sect.\ref{sec2}, we present  some basic estimates for the perturbed equation (\ref{eq11}). In Sect.\ref{sec3}, we prove  Theorem \ref{B} in detail.

\section{Preliminary}\label{sec2}
Let $(E,\bp_E,\phi)$ be a Higgs bundle over $X$ and $H$ be a Hermitian metric on $E.$ Set
$$\therm(E,H)=\{\eta\in \tend(E)\mid \eta^{*H}=\eta\}$$
and $$\therm^+(E,H)=\{ \rho\in \therm(E,H)\mid H\rho \ \text{is positive definite}\}.$$
Suppose  $f\in \therm^+(E,H_0)$ is a solution of the equation (\ref{eq11}) with the background metric $H_0$ for some  $\vps\in (0,1]$. Substituting
\begin{equation*}
\mfk_H=\mfk_{H_0}+\sqrt{-1}\Lambda_{\omega}\left(\bp(f^{-1}\circ\pt_{H_0} f)+[\phi,\phi^{*H}-\phi^{*H_0}]\right)
\end{equation*}
into (\ref{eq11}), we obtain
\begin{eqnarray}\label{eq21}
L_{\vps}(f)=\mathcal{K}_{H_0}-\lmd \mbox{Id}_E+\sqrt{-1}\Lambda_{\omega}\left(\bp(f^{-1}\circ\pt_{H_0} f)+[\phi,\phi^{*H}-\phi^{*H_0}]\right)+\vps \log f=0.
\end{eqnarray}
 Furthermore, by an appropriate  conformal change,  we can assume that $H_0$ satisfies
$$\mbox{tr}(\mathcal{K}_{H_0}-\lmd\mbox{Id}_{E})=0.$$
In fact, let $H_0=e^{\vph}H_0',$ where $H_0'$ is an arbitrary metric and $\vph$ is a smooth function satisfying
\begin{equation}\label{eq22}
\sqrt{-1}\Lambda_{\og}\bp\partial(\vph)=-\frac{1}{r}\mathrm{tr}(\mathcal{K}_{H_0'}-\lmd\cdot\mathrm{Id}_E).
\end{equation}
 Since $\int_X\mbox{tr}(\mathcal{K}_{H_0'}-\lambda\mbox{Id}_E)\omega^n=0,$ equation (\ref{eq22}) is solvable.

 For simplicity, we  set $ \Phi(H, \phi)=\mathcal{K}_{H}-\lmd \cdot \mathrm{Id}_E$. It is easily to check that $\Phi(H, \phi)^{*H}=\Phi(H, \phi).$ The following two lemmas are proved by Teleman and L\"ubke in \cite{MA}. Here we present the proofs just for readers' convenience.
\begin{lemma}\label{lm1}
 Fix a background Hermitian metric $H_0$ satisfying $\tr\, \Phi(H_0, \phi)=0.$ Then for any $f\in \mbox{Herm}^+(E,H_0)$ such that $ L_{\vps}(f)=0,$ it holds
\begin{align*}
P(\mbox{tr}\log f)+\vps\mbox{tr}\log f=0,
\end{align*}
where $P$ is denoted by $P=\sqrt{-1}\Lambda_{\og}\bp\partial.$  Furthermore, we have $\det f=1.$
\end{lemma}

\begin{proof}

By $ \partial \log \det f=\mathrm{Tr}(f^{-1}\partial f) $ and $ \log \det f=\mathrm{tr}\log f,$ we have
  \begin{equation}\label{eq23}
  \begin{split}
  \mathrm{Tr}\sqrt{-1}\wedge_{\og}\left( \bp(f^{-1}\pt_{H_0}f)\right) =&\sqrt{-1}\wedge_{\og}\bp \mathrm{Tr} (f^{-1}\partial f)\\
  =&\sqrt{-1}\wedge_{\og}\bp\partial \log\det f\\
  =&\sqrt{-1}\wedge_{\og}\bp\partial \mathrm{tr}\log f.
  \end{split}
  \end{equation}

Then combining (\ref{eq23}) with $\mathrm{Tr}\sqrt{-1}\wedge_{\og}[\phi, f^{-1}\phi^{*H_0}f-\phi^{H_0}]=0$, we conclude that
  \begin{equation*}
  \begin{split}
  0=&\mathrm{tr} L_{\vps}(f)\\
  =&\mbox{tr}\Phi(H_0, \phi)+\mbox{tr}\sqrt{-1}\wedge_{\og}\left(\bp(f^{-1}\pt_{H_0} f)\right)+\vps \mbox{tr}\log f\\
  =& P(\mbox{tr}\log f)+\vps \mbox{tr}\log f.
  \end{split}
  \end{equation*}
Furthermore, by the maximum principle, we have $\mbox{tr}\log f =0$ and $\det f=1.$
\end{proof}

\begin{lemma}\label{lm2}
If $f\in \mathrm{Herm}^+(E,H_0)$ satisfies $L_{\vps}(f)=0$ for some $\vps>0$,
then there holds that
\begin{enumerate}
  \item $\frac{1}{2}P\left(\mid \log{f}\mid_{H_0}^2\right)+\vps \mid \log{f}\mid_{H_0}^2\leq \mid\Phi(H_0, \phi)\mid_{H_0} \mid \log{f}\mid_{H_0};$
  \item  $m=\max_X\mid \log{f} \mid_{H_0}\leq \frac{1}{\vps}\cdot \max_X \mid \Phi(H_0, \phi)\mid_{H_0} $;
  \item $m\leq C\cdot (\parallel \log{f} \parallel_{L^2}+\max_X \mid \Phi(H_0, \phi)\mid_{H_0})$, where $C$ only depends on $g$ and $X$.
\end{enumerate}

\end{lemma}

\begin{proof}

 (\romannumeral1) Taking the point-wise inner product with $\log f$  respect to $H_0$ of both sides of
 (\ref{eq21}),
we have
\begin{eqnarray}\label{eq24}
\begin{split}
&\lag\sqrt{-1}\Lambda_{\omega}\bp(f^{-1}\circ\pt_{H_0} f),\log{f}\rag_{H_0}+\lag\sqrt{-1}\Lambda_{\omega}[\phi,\phi^{*H}-\phi^{*H_0}],\log{f}\rag_{H_0}\\
&+\vps \mid \log{f}\mid_{H_0}^2=-\lag \Phi(H_0,\phi),\log{f}\rag_{H_0}.
\end{split}
\end{eqnarray}
Set
$A=\lag\sqrt{-1}\Lambda_{\omega}\bp(f^{-1}\circ\pt_{H_0} f),\log{f}\rag_{H_0}$ and $B=\lag\sqrt{-1}\Lambda_{\omega}[\phi,\phi^{*H}-\phi^{*H_0}],\log{f}\rag_{H_0}.$  From the result in \cite[p.~74]{MA}, we have
\begin{eqnarray}\label{eq25}
P(\mid\log{f}\mid_{H_0}^2)\leq 2A.
\end{eqnarray}
Now we estimate $B$. Let $H(t)=H_0 e^{ts},$ $t\in [0,1]$ be a curve in $\therm^+(E)$ connecting $H_0$ and $H_0f$, where $s=\log{f}.$  Set   $\zeta(t)=H_0\left(\sqrt{-1}\Lambda_{\og}\left[\phi,e^{-ts}\phi^{*H_0}e^{ts}\right],s\right).$ The $t$-derivative of $\zeta(t)$ is
\begin{eqnarray*}
\frac{\mathrm{d}}{\mathrm{dt}}\zeta(t)=&\langle  \sqrt{-1}\Lambda_{\og }\left[\phi,-se^{-ts}\phi^{*H_0}e^{ts}+e^{-ts}\phi^{*H_0}e^{ts}s\right], s \rangle_{H_0} =\mid [s,e^{\frac{ts}{2}}\phi e^{-\frac{ts}{2}}]   \mid^2_{H_0}\geq 0.
\end{eqnarray*}
This implies
\begin{equation}\label{eq26}
B=\zeta(1)\geq\zeta(0)=0.
\end{equation}
By (\ref{eq24}-\ref{eq26}), we have
\begin{eqnarray*}
\frac{1}{2}P(\mid \log{f} \mid_{H_0}^2)+\vps\mid \log{f} \mid^2_{H_0}\leq- \lag\Phi(H_0, \phi),\log{f}\rag_{H_0}\leq\mid \Phi(H_0,\phi) \mid_{H_0}\mid \log{f}\mid_{H_0}.
\end{eqnarray*}

(\romannumeral2) Assuming  $\mid \log{f} \mid_{H_0}^2$ attains its maximum at $p\in X$,  we have
\begin{eqnarray*}
0\leq \frac{1}{2}P(\mid \log{f} \mid^2)(p)\leq \left(\mid \Phi(H_0,\phi) \mid_{H_0}(p)-\vps \mid \log{f} \mid_{H_0}(p)\right)\mid \log{f} \mid_{H_0}(p).
\end{eqnarray*}
Then it follows that
\begin{eqnarray*}
\max_X\mid \log{f} \mid_{H_0}=\mid \log{f} \mid_{H_0}(p)\leq \frac{1}{\vps}\mid \Phi(H_0,\phi) \mid_{H_0}(p)\leq \frac{1}{\vps} \max\limits_X\mid \Phi(H_0,\phi)\mid_{H_0}.
\end{eqnarray*}

 (\romannumeral3)   From (\romannumeral1), we get
\begin{equation*}
P(\mid \log f\mid_{H_0}^2) \leq \mid \Phi(H_0, \phi)\mid^2_{H_0}+\mid \log f\mid_{H_0}^2\leq \max_X \mid \Phi(H_0, \phi)\mid^2_{H_0}+\mid \log f\mid_{H_0}^2.
\end{equation*}
Then by  Moser's iteration,  there exist a constant $C>0$ depending on $g$ and $X$  such that
$$m\leq C\cdot (\parallel \log{f} \parallel_{L^2}+\max_X \mid \Phi(H_0, \phi)\mid_{H_0}).$$
\end{proof}

\section{Proof of Theorem \ref{B}}\label{sec3}

Before we give the detailed proof, we recall some notation. Given $\eta \in \therm(E,H),$ from  \cite[p.~237]{MA}, we can choose an open dense subset $W\subseteq X$ satisfying at each $x\in W$ there exist an open neighbourhood $U$ of $x$, a local unitary basis  $\{e_a\}_{a=1}^{r}$ respect to $H$ and functions $\{\lmd_a\in C^\infty(U,R)\}_{a=1}^{r}$ such that
\begin{eqnarray*}
\eta(y)=\sum_{a=1}^r \lmd_a(y)\cdot e_a(y)\otimes e^a(y)
\end{eqnarray*}
for all $y\in U$, where  $\{e^a\}_{a=1}^{r}$ denotes the dual basis of $E^*$. Let $\vph\in C^{\infty}(R,R),$   $\Psi\in C^{\infty}(R\times R, R)$ and $A=\sum^r_{a,b=1}A^b_a e^a\otimes e_b\in \tend(E).$ We denote $\vph(\eta)$ and $\Psi(\eta)(A)$ by
\begin{equation}\label{eq31}
\vph(\eta)(y)=\sum^r_{a=1}\vph(\lmd_a)e_a\otimes e^a
\end{equation}
and
\begin{equation}\label{eq32}
\Psi(\eta)(A)(y)=\Psi(\lmd_a,\lmd_b)A_a^b e^a\otimes e_a.
\end{equation}

\begin{prop}\label{prop1}
If $f\in \therm^+(E,H_0)$ solves (\ref{eq21}) for some $\vps$, then there holds
\begin{equation}\label{eq33}
\int_X \tr(\Phi(H_0,\phi) s)\frac{\og^n}{n!}+\int_X\lag\Psi(s)(\mathcal{D}''s), \mathcal{D}''s\rag_{H_0}\frac{\og^n}{n!}=-\vps \pal s \pal^2_{L^2},
\end{equation}
where $s=\log f$, $\mfd'' =\bp_E+ \phi$ and
\begin{equation*}
\Psi(x,y)=
\begin{cases}
&\frac{e^{y-x}-1}{y-x},\ \ \ x\neq y;\\
&\ \ \ \  1,\ \ \ \ \ \  x=y.
\end{cases}
\end{equation*}
\end{prop}

\begin{proof}
First, (\ref{eq21}) gives
\begin{equation}\label{eq34}
\begin{split}
&\int_X \tr(\Phi(H_0,\phi) s)\frac{\og^n}{n!}+\int_X\lag\sqrt{-1}\Lambda_{\og}\bp(f^{-1}\pt_{H_0}f), s\rag_{H_0}\frac{\og^n}{n!}\\
+&\int_X \lag\sqrt{-1}\Lambda_{\og}[\phi, \phi^{*H}-\phi^{*H_0}], s\rag_{H_0}\frac{\og^n}{n!}
+\vps\pal s\pal^2_{L^2}=0,
\end{split}
\end{equation}
where $H=H_0 f.$
Then comparing (\ref{eq34}) with (\ref{eq33}), it is sufficient to show
\begin{equation}\label{eq35}
\int_X \lag\sqrt{-1}\Lambda_{\og}\bp(f^{-1}\pt_{H_0}f)+[\phi, \phi^{*H}-\phi^{*H_0}], s\rag_{H_0}\frac{\og^n}{n!}=\int_X\lag\Psi(s)(\mathcal{D}''s), \mathcal{D}''s\rag_{H_0}\frac{\og^n}{n!}.
\end{equation}
We will divide the proof of (\ref{eq35}) into the following two steps.

\medskip
\emph{Step 1} We show that
\begin{equation}\label{eq36}
   \int_X \lag\sqrt{-1}\Lambda_{\og}\bp(f^{-1}\pt_{H_0}f)+[\phi, \phi^{*H}-\phi^{*H_0}], s\rag_{H_0}\frac{\og^n}{n!}=\int_X \mathrm{Tr}\sqrt{-1}\Lambda_{\og}\{f^{-1}\mathcal{D}'f\wedge \mathcal{D}''s\}\frac{\og^n}{n!},
\end{equation}
 where $\mfd'=\pt_{H_0}+\phi^{*H_0}$.

 By using Stokes formula, we have
\begin{equation}\label{eq37}
\begin{split}
  &\int_X \lag\sqrt{-1}\Lambda_{\og}\bp(f^{-1}\pt_{H_0}f), s\rag_{H_0}\frac{\og^n}{n!}\\
=&\int_X\bp\left(\mathrm{Tr}\{ \sqrt{-1}f^{-1}(\pt_{H_0} f) s\}\frac{\og^{n-1}}{(n-1)!}\right)+\int_X\mathrm{Tr}\{\sqrt{-1}f^{-1}\pt_{H_0} f \bp s\}\frac{\og^{n-1}}{(n-1)!}\\
&+\int_X \mathrm{Tr}\{\sqrt{-1}f^{-1}(\pt_{H_0} f) s\}\bp\frac{\og^{n-1}}{(n-1)!} \\
=&\int_X\mathrm{Tr}\{\sqrt{-1}f^{-1}\pt_{H_0} f \bp s\}\frac{\og^{n-1}}{(n-1)!}+\int_X \mathrm{Tr}\{\sqrt{-1}f^{-1}(\pt_{H_0} f) s\}\bp\frac{\og^{n-1}}{(n-1)!}.
\end{split}
\end{equation}
 Since $sf=fs$, it follows that
 \begin{equation}\label{eq38}
 \begin{split}
 \mathrm{Tr}\left(f^{-1}(\pt_{H_0} f) s\right)=&\mathrm{Tr}\left(f^{-1}(\pt_{H_0} \sum_{k=1}^{+\infty}\frac{s^k}{k!}) s\right)
 =\mathrm{Tr}\left(f^{-1}\sum_{k=1}^{+\infty}\sum_{j=0}^{k-1}\displaystyle{\frac{s^j(\pt_{H_0}s) s^{k-1-j}}{k!}} s\right)\\
 =&\mathrm{Tr}\left(f^{-1}\sum_{k=1}^{+\infty}\sum_{j=0}^{k-1}\displaystyle{\frac{s^k\pt_{H_0}s }{k!}}\right)
 =\mathrm{Tr}\left(f^{-1}\sum_{k=1}^{+\infty}\displaystyle{\frac{s^k\pt_{H_0}s }{(k-1)!}}\right)\\
 =&\mathrm{Tr}\left(f^{-1}\sum_{k=1}^{+\infty}\displaystyle{\frac{s^{k-1}}{(k-1)!}}s\pt_{H_0}s\right)
 =\mathrm{Tr}( s\pt_{H_0}s).
 \end{split}
 \end{equation}
(\ref{eq38}) together with $\pt\bp \og^{n-1}=0$ gives
\begin{equation}\label{eq39}
\begin{split}
\int_X \mathrm{Tr}(\sqrt{-1}f^{-1}(\pt_{H_0} f) s)\bp \frac{\og^{n-1}}{(n-1)!}=\int_X \frac{1}{2} \pt\mathrm{Tr}(\sqrt{-1}s^2)\bp \frac{\og^{n-1}}{(n-1)!}=0.
\end{split}
\end{equation}
From (\ref{eq37}) and (\ref{eq39}), we have
\begin{equation}\label{eq310}
\int_X \lag\sqrt{-1}\Lambda_{\og}\bp(f^{-1}\pt_{H_0}f), s\rag_{H_0}\frac{\og^n}{n!}=\int_X\mathrm{Tr}\{\sqrt{-1}f^{-1}\pt_{H_0} f \bp s\}\frac{\og^{n-1}}{(n-1)!}.
\end{equation}
Then noticing that $\tr(AB)=(-1)^{pq} \tr(BA),$ where $A$ is an $\tend(E)$ valued $p$-form and $B$ is an $\tend(E)$ valued $q$-form, there holds
\begin{equation}\label{eq311}
\begin{split}
&\int_X \mathrm{Tr}\{\sqrt{-1}\Lambda_{\og}\left[\phi, f^{-1}\phi^{*H_0}f-\phi^{*H_0}\right]s \}\frac{\og^n}{n!}\\
=&\int_X \sqrt{-1} \mathrm{Tr}\{ \phi f^{-1}\phi^{*H_0}fs+f^{-1}\phi^{*H_0}f\phi s-(\phi\phi^{*H_0}s+\phi^{*H_0}\phi s)\}\frac{\og^{n-1}}{(n-1)!}\\
=&\int_X \sqrt{-1}\mathrm{Tr}\{ -f^{-1}\phi^{*H_0} fs\phi+f^{-1}\phi^{*H_0} f \phi s+\phi^{*H_0}s\phi-\phi^{*H_0}\phi s\}\frac{\og^{n-1}}{(n-1)!}\\
=&\int_X\sqrt{-1}\mathrm{Tr}\{ f^{-1}[\phi^{*H_0}, f][\phi, s]\}\frac{\og^{n-1}}{(n-1)!}.
\end{split}
\end{equation}
Therefore, we complete Step 1 by substituting (\ref{eq310}) and (\ref{eq311}) into the left hand side of (\ref{eq36}).

\emph{Step 2} We show that
 \begin{equation}\label{eq312}
\mathrm{Tr}\sqrt{-1}\Lambda_{\og}\{f^{-1}\mathcal{D}'f\wedge \mathcal{D}''s\}=\lag \Psi(s)(\mfd'' s), \mfd'' s\rag_{H_0}
\end{equation}
 holds on $X$.

 From \cite[p.~237-238]{MA}, there exists an open dense subset $W\subseteq X$ such that at each $x\in W$, one has
\begin{equation*}
\mathcal{D}'f(x)=e^{\lmd_a}\pt\lmd_a e_a\otimes e^a+(e^{\lmd_b}-e^{\lmd_a})(A_b^a+\ol{\phi_a^b})e_a\otimes e^b
\end{equation*}
and
\begin{equation*}
\mathcal{D}''s(x)=\bp\lmd_a e_a\otimes e^a+(\lmd_b-\lmd_a)\left(  -\ol{A^b_a}+\phi_b^a \right)e_a\otimes e^b,
\end{equation*}
where $\{e_a\}^r_{a=1}$ is a local unitary basis of $E$ respect to $H_0$ and the $(1,0)$-forms $A_a^b$ are defined by $\pt_{H_0}e_a=A_a^b e_b.$

It follows that at each $x\in W$,
\begin{equation*}
\begin{split}
&\mathrm{Tr}\sqrt{-1}\Lambda_{\og}\{f^{-1}\mathcal{D}'f \wedge \mathcal{D}''s\}\\
=&\sum_{a=1}^{r}\mid \bp\lmd_a \mid^2+
\sum_{a\neq b}(e^{\lmd_b-\lmd_a}-1)(\lmd_a-\lmd_b)\sqrt{-1}\Lambda_{\og}(A_b^a+\ol{\phi_a^b})\wedge(-\ol{A_b^a}+\phi_a^b)\\
=&\sum^r_{a=1}\mid \bp \lmd_a\mid^2+\sum_{a\neq b}\frac{e^{\lmd_b-\lmd_a}-1}{\lmd_b-\lmd_a}(\lmd_b-\lmd_a)^2\mid -\ol{A_b^a}+\phi_a^b\mid^2\\
=&\sum_{a,b}\Psi(\lmd_a,\lmd_b)\mid(\mfd''s)^b_a\mid^2.
\end{split}
\end{equation*}
We now turn to calculating the right hand side of (\ref{eq312}). By the construction (\ref{eq32}), we have
\begin{equation*}
\begin{split}
\Psi(s)(\mfd''s)=&\bp\lmd_a e_a\otimes e^a+\Psi(\lmd_a, \lmd_b)(\lmd_a-\lmd_b)(-\ol{A_b^a}+\phi_a^b)e_b\otimes e^a\\
=&\sum^r_{a=1}\bp\lmd_a e_a\otimes e^a+\sum_{a\neq b}\frac{e^{\lmd_b-\lmd_a}-1}{\lmd_b-\lmd_a}(\lmd_a-\lmd_b)(-\ol{A_b^a}+\phi_a^b)e_b\otimes e^a
\end{split}
\end{equation*}
Then at each $x\in W$ there holds
\begin{equation*}
\begin{split}
\lag \Psi(s)(\mathcal{D}''s),  \mathcal{D}''s  \rag_{H_0}
=&\sum^r_{a=1}\mid \bp\lmd_a \mid^2+\sum_{a\neq b}(e^{\lmd_b-\lmd_a}-1)(\lmd_b-\lmd_a)\mid -\ol{A_b^a}+\phi_a^b \mid^2\\
=&\sum_{a,b}\Psi(\lmd_a,\lmd_b)\mid(\mfd''s)^b_a\mid^2\\
=&\mathrm{Tr}\sqrt{-1}\Lambda_{\og}\{f^{-1}\mathcal{D}'f\wedge \mathcal{D}''s\}.
\end{split}
\end{equation*}

This forces
\begin{equation}\label{eq313}
 \lag \Psi(s)(\mathcal{D}''s),  \mathcal{D}''s  \rag_{H_0}=\mathrm{Tr}\sqrt{-1}\Lambda_{\og}\{f^{-1}\mathcal{D}'f\wedge \mathcal{D}''s\}
 \end{equation}
 holds on $X.$
So, combining (\ref{eq36}) with (\ref{eq313}) we have (\ref{eq35}).
\end{proof}
Then, we prove the ``only if '' part of Theorem \ref{B}. In fact, we prove the following theorem
\begin{thm}\label{A.5}
If Higgs bundle $(E, \bp_E, \phi)$ is $\og$-semi-stable, then  $\max\limits_X \mid \Phi(H_{\vps}, \phi)\mid_{H_{\vps}}\rightarrow 0$ as $\vps\rightarrow 0.$
\end{thm}

\begin{proof}
Let $\{f_{\vps}\}_{0<\vps\leq 1}$ be the solutions of equation (\ref{eq21}) with background metric $H_0.$ Then there holds that
\begin{equation*}
\pal \log f_{\vps}\pal^2_{L^2}=-\frac{1}{\vps}\int_X \langle \Phi(H_{\vps},\phi), \log f \rangle_{H_{\vps}}\displaystyle{\frac{\og^n}{n!}}.
\end{equation*}

\textbf{Case 1}, $\exists C_1>0$ such that $\pal \log{f_{\vps}}\pal_{L^2}< C_1<+\infty $.
From Lemma \ref{lm2}, we have
\begin{equation*}
\max\limits_X \mid \Phi(H_{\vps}, \phi)\mid_{H_{\vps}}=\vps\cdot\max\limits_X \mid \log f_{\vps}\mid_{H_{\vps}}<\vps C\cdot(C_1+\max\limits_X\mid\Phi(H_0,\phi)\mid_{H_0}).
\end{equation*}
Then it follows that $\max\limits_X \mid \Phi(H_{\vps}, \phi)\mid_{H_{\vps}}\rightarrow 0$ as $\vps\rightarrow 0.$

\textbf{Case 2}, $\overline{\lim\limits_{\vps\rightarrow 0}}\pal \log{f_{\vps}}\pal_{L^2}\rightarrow \infty.$

\textbf{Claim}  If $(E,\bp_E, \phi)$ is semi-stable, there holds
\begin{equation}\label{eq314}
\lim\limits_{\vps\rightarrow 0}\max\limits_X \mid \Phi(H_{\vps}, \phi)\mid_{H_{\vps}}=\lim\limits_{\vps\rightarrow 0}\vps \max\limits_X \mid \log f_{\vps}\mid_{H_{\vps}}=0.
\end{equation}

We will follow Simpson's argument (\cite[Proposition~5.3]{simpson1988constructing}) to show that if the claim does not hold, there exists a Higgs subsheaf contradicting the semi-stability.

If the claim does not hold,  then there exist $\delta>0$ and a subsequence $\vps_i \raw 0,\ i\raw +\infty$, such that
 $$\pal\log f_{\vps_i}\pal_{L^2}\raw +\infty$$
 and
\begin{equation}\label{a.3}
\max\limits_X \mid \Phi(H_{\vps_i},\phi) \mid_{H_{\vps_i}}=\vps_i \max\limits_X \mid \log f_{\vps_i}\mid_{H_{\vps_i}}\geq \delta .
\end{equation}
Setting $s_{\vps_i}=\log f_{\vps_i},$ $l_i=\pal s_{\vps_i} \pal_{L^2}$ and $u_{\vps_i}=s_{\vps_i}/l_{i}$, it follows that $\tr u_{\vps_i}=0$ and $\pal u_{\vps_i} \pal_{L^2}=1.$ Then combining (\ref{a.3}) with Lemma 2.2 (\romannumeral3), we have
\begin{equation}\label{a.4}
l_i\geq \frac{\delta}{C\vps_i}-\max\limits_X \mid \Phi(H_0,\phi) \mid_{H_0}.
\end{equation}
and
\begin{equation}\label{a.5}
\max\limits_X\mid u_{\vps_i}\mid<\frac{C}{l_i}(l_i+\max\limits_X \mid \Phi(H_0,\phi) \mid)<C_2<+\infty.
\end{equation}

\medskip
$Step\  1$  We show that $\pal u_{\vps_i} \pal_{L^2_1}$ are uniformly bounded. Since  $\pal u_{\vps_i} \pal_{L^2}=1$, we only need to  prove $\pal \mathcal{D}''u_{\vps_i} \pal_{L^2}$ are uniformly bounded.

By (\ref{eq11}) and Proposition \ref{prop1}, for each $f_{\vps_i},$ there holds
\begin{equation}\label{eq315}
  \int_X\mathrm{Tr}\{\Phi(H_0,\phi) u_{\vps_i}\}\frac{\og^n}{n!}+l_i\int_X \lag\Psi(l_{i}u_{\vps_i})(\mfd''u_{\vps_i}), \mfd''u_{\vps_i}\rag_{H_0}\frac{\og^n}{n!}=-\vps_i l_i
 \end{equation}
Substituting (\ref{a.4}) into (\ref{eq315}), we have
\begin{equation}\label{eq3160}
C^*+\int_X\mathrm{Tr}\{\Phi(H_0,\phi) u_{\vps_i}\}+\lag l_i\Psi(l_{i}u_{\vps_i})(\mfd''u_{\vps_i}), \mfd''u_{\vps_i}\rag_{H_0}\frac{\og^n}{n!}\leq \vps_i\max\limits_X\mid\Phi(H_0,\phi)\mid_{H_0},
\end{equation}
where $C^*=\frac{\delta}{C}$.

 Consider the function
\begin{equation}
l\Psi(lx,ly)=
\begin{cases}
\ \ \ \ l,\ \ &x=y;\\
\frac{e^{l(y-x)}-1}{y-x},\ \ &x\neq y.
\end{cases}
\end{equation}
 From (\ref{a.5}), we may assume that $(x,y)\in [-C_2,C_2]\times[-C_2,C_2].$ It is easy to check that

\begin{equation}\label{a1}
l\Psi(lx,ly)\lraw
\begin{cases}
(x-y)^{-1},\ \ \ &x>y;\\
\ \ +\infty,\ \ \ &x\leq y,
\end{cases}
\end{equation}
increases monotonically as $l\raw +\infty.$
 Let $ \zeta \in C^{\infty} (R\times R, R^+)$ satisfying $\zeta(x,y)<(x-y)^{-1}$ whenever $x>y.$ From (\ref{eq3160}), (\ref{a1}) and the arguments in Lemma 5.4(\cite{simpson1988constructing}), we have
 \begin{equation}\label{a.2}
 C^*+\int_X\tr\{u_{\vps_i}\Phi(H_0,\phi)\}+\lag\zeta(u_{\vps_i})\mfd''u_{\vps_i},\mfd''u_{\vps_i}\rag_{H_0}\frac{\og^n}{n!}\leq \vps_i\max\limits_X\mid\Phi(H_0,\phi)\mid_{H_0},
 \end{equation}
when $i\gg 0.$ Particularly, we take $\zeta(x,y)=\frac{1}{2C_2}.$ It is obvious that when $(x,y)\in [-C_2,C_2]\times[-C_2,C_2]$ and $x>y,$ $\frac{1}{3C_2}<\frac{1}{x-y}.$ This implies that
$$ C^*+\int_X\tr\{u_{\vps_i}\Phi(H_0,\phi)\}+\frac{1}{3C_2}\mid\mfd''u_{\vps_i}\mid^2_{H_0}\frac{\og^n}{n!}\leq \vps_i\max\limits_X\mid\Phi(H_0,\phi)\mid_{H_0},$$
when $i\gg 0.$ Then we have
$$\int_X\mid\mfd''u_{\vps_i}\mid^2_{H_0}\frac{\og^n}{n!}\leq 3C_2 \max\limits_X\mid\Phi(H_0,\phi)\mid_{H_0}( \mathrm{Vol}(X)^{\frac{1}{2}}+1).$$
 Thus, $u_{\vps_i}$ are bounded in $L_1^2.$ We can choose  subsequence $\{u_{\vps_{i_j}}\}$ such that $u_{\vps_{i_j}}\rightharpoonup u_{\infty}$ weakly in $L^2_1.$  We still write it $\{u_{\vps_i}\}^{\infty}_{i=1}$ for simplicity. Noting that $L_1^2\hookrightarrow L^2,$ we have
 $$1=\int_X \mid u_{\vps_i}\mid^2_{H_0}\raw \int_X \mid u_{\infty}\mid^2_{H_0}.$$
 This indicates that $\pal u_{\infty}\pal_{L^2}=1$ and $u_{\infty}$ is nontrivial.

So by (\ref{a.2}) and the same discussion in Lemma 5.4 (\cite{simpson1988constructing}), there holds
\begin{equation}\label{a.6}
 C^*+\int_X\tr\{u_{\infty}\Phi(H_0,\phi)\}+\lag\zeta(u_{\infty})\mfd''u_{\infty},\mfd''u_{\infty}\rag_{H_0}\frac{\og^n}{n!}\leq 0.
 \end{equation}

\medskip
$Step\  2$ By using Uhlenbeck and Yau's trick in \cite{UhYau} to construct a Higgs sub-sheaf which contradicts the semi-stability of $E$.

From (\ref{a.6}) and the technique in Lemma 5.5 in (\cite{simpson1988constructing}), we have the eigenvalues of $u_{\infty}$ are constant almost everywhere. Let $\mu_1<\mu_2<\cdots\mu_l$ be the distinct eigenvalues of $u_{\infty}$. The facts that  $\tr(u_{\infty})=\tr(u_{\vps_i})=0$ and $\pal u_{\infty}\pal_{L^2}=1$  force $2\leq l\leq r.$
For each $\mu_{\alpha}\ (1\leq\alpha\leq l-1),$ we construct a function $P_{\alpha}: R\lraw R$ such that
\begin{equation*}
P_{\alpha}=
\begin{cases}
&1,\ \ \ \ \ x\leq \mu_{\alpha}\\
&0,\ \ \ \ \ x\geq \mu_{\alpha+1}.
\end{cases}
\end{equation*}
Setting $\pi_{\alpha}=P_{\alpha}(u_{\infty}),$ from \cite[p.~887]{simpson1988constructing}, we have
\begin{enumerate}
  \item $\pi_{\alpha}\in L^2_1;$
  \item $\pi^2_{\alpha}=\pi_{\alpha}=\pi_{\alpha}^{*H_0};$
  \item $(\mathrm{Id}-\pi_{\alpha})\bp \pi_{\alpha}=0;$
  \item $(\mathrm{Id}-\pi_{\alpha})[\phi, \pi_{\alpha}]=0.$
\end{enumerate}

By Uhlenbeck and Yau's regularity statement of $L^2_1$-subbundle (\cite{UhYau}), $\{\pi_{\alpha}\}^{l-1}_{\alpha=1}$ determine $l-1$  Higgs sub-sheaves of $E$. Set $E_{\alpha}=\pi_{\alpha}(E).$  Since $\mathrm{tr}u_{\infty}=0$ and
$u_{\infty}=\mu_l Id-\sum^{l-1}_{\alpha=1}(\mu_{\alpha+1}-\mu_{\alpha})\pi_{\alpha},$
there holds
\begin{eqnarray}\label{eq318}
\mu_{l}\mathrm{rank}{E}=\sum^{l-1}_{\alpha=1}(\mu_{\alpha+1}-\mu_{\alpha})\mathrm{rank}{E_{\alpha}},
\end{eqnarray}

Construct
\begin{equation*}
\nu=\mu_l\deg(E)-\sum^{l-1}_{\alpha=1}(\mu_{\alpha+1}-\mu_{\alpha})\deg(E_{\alpha}).
\end{equation*}

From one hand, substituting (\ref{eq318}) into $\nu$,
\begin{equation}\label{eq319}
\nu=\sum^{l-1}_{\alpha=1}(\mu_{\alpha+1}-\mu_{\alpha})\mathrm{rank}{E_{\alpha}}\left( \displaystyle{\frac{\deg(E)}{\mathrm{rank}{E}}}-
\displaystyle{\frac{\deg(E_{\alpha})}{\mathrm{rank}{E_{\alpha}}}}\right)
\end{equation}

From the other hand,  substituting the Chern-Weil formula (Prop. 2.3 in \cite{bruasse})
\begin{equation*}
\deg(E_{\alpha})=\int_{X}\mathrm{Tr}(\pi_{\alpha}\mfk_{H_0})-\mid \mfd^{''}\pi_{\alpha}\mid^2 \displaystyle{\frac{\og^n}{n!}}
\end{equation*}
into $\nu$,
\begin{equation*}
\begin{split}
\nu=&\mu_l \int_X\mathrm{Tr}(\mfk_{H_0})-
\sum^{l-1}_{\alpha=1}(\mu_{\alpha+1}-\mu_{\alpha})\left\{\int_X\mathrm{Tr}(\pi_{\alpha}\mfk_{H_0})
-\int_X \mid \mfd^{''}\pi_{\alpha} \mid^2_{H_0}\right\}\\
=&\int_X\mathrm{Tr}\left( (\mu_l\mathrm{Id}-\sum^{l-1}_{\alpha=1}(\mu_{\alpha+1}-\mu_{\alpha})\pi_{\alpha} \right)\mfk_{H_0}
+\sum^{l-1}_{\alpha=1}(\mu_{\alpha+1}-\mu_{\alpha})\int_X  \mid \mathcal{D}^{''}\pi_{\alpha} \mid^2\\
=& \int_X \mathrm{Tr}(u_{\infty}\mfk_{H_0})+\lag \sum^{l-1}_{\alpha=1}(\mu_{\alpha+1}-\mu_{\alpha})
(dP_{\alpha})^2(u_{\infty})(\mfd^{''}u_{\infty}), \mfd''u_{\infty}\rag_{H_0},
\end{split}
\end{equation*}
where the function $dP_{\alpha}: R\times R\lraw R$ is defined by
\begin{equation*}
dP_{\alpha}(x,y)=
\begin{cases}
\displaystyle{\frac{P_{\alpha}(x)-P_{\alpha}(y)}{x-y}},\qquad &x\neq y;\\
\qquad P_{\alpha}'(x),& x=y.
\end{cases}
\end{equation*}
From simple calculation, we have if $\mu_{\beta}\neq \mu_{\gamma}$
\begin{equation}
\sum^{l-1}_{\alpha=1}(\mu_{\alpha+1}-\mu_{\alpha})
(dP_{\alpha})^2(\mu_{\beta}, \mu_{\gamma})=\mid \mu_{\beta}-\mu_{\gamma} \mid^{-1}.
\end{equation}

Since $\tr u_{\infty}=0,$ so by (\ref{a.6}) and the same arguments in \cite[p.~793-794]{li2012existence} there holds
\begin{equation}\label{eq320}
\nu=\int_X \mathrm{Tr}(u_{\infty}\Phi(H_0,\phi))+\lag \sum^{l-1}_{\alpha=1}(\mu_{\alpha+1}-\mu_{\alpha})
(dP_{\alpha})^2(u_{\infty})(\mfd^{''}u_{\infty}), \mfd''u_{\infty}\rag_{H_0}<-C^*.
\end{equation}
Combining (\ref{eq319}) with (\ref{eq320}), we have
$$\sum^{l-1}_{\alpha=1}(\mu_{\alpha+1}-\mu_{\alpha})\mathrm{rank}{E_{\alpha}}\left( \displaystyle{\frac{\deg(E)}{\mathrm{rank}{E}}}-
\displaystyle{\frac{\deg(E_{\alpha})}{\mathrm{rank}{E_{\alpha}}}}\right)<0.$$
This indicates there must exist a term $(\mu(E)-\mu(E_{\alpha_0}))<0$, which contradicts the semi-stability of $E$.
\end{proof}

Finally, we  prove the ``if '' part of the Theorem \ref{B}.
\begin{thm}\label{tm2}
Let $(X,\og)$ be an $n$-dimensional compact Gauduchon manifold and $(E,\bp_E, \phi)$ be a Higgs bundle over $X$. If  $(E, \bp_E, \phi)$ admits an approximate Hermitian-Einstein manifold, then  $(E, \bp_E, \phi)$ is $\og$-semi-stable.
\end{thm}

Firstly, following the techniques of Kobayahsi \cite{Kobayashi1} and Bruzzo-Gra\~na Otero's \cite{bruzzo}, we  prove a Higgs version vanishing theorem.

\begin{prop}\label{prop2}
Let $(X, \og)$ be an $n$-dimensional Hermitian manifold with Gauduchon metric $\og$ and $(E, \bp_E, \phi)$ be a Higgs bundle over $X$. Assume that $E$ admits an approximate Hermitian Einstein structure. If  $\deg E <0$ , then $E$ has no nonzero $\phi$-invariant sections of $E$.
\end{prop}

\begin{proof}
Let $H$ be a Hermitian metric over $E$ and  $s$ be a $\phi$-invariant holomorphic  section of $E$.  From simple calculation, one has
\begin{equation*}
H(s, \sqrt{-1}\Lambda_{\og}[\phi, \phi^{*H}]s)\geq 0.
\end{equation*}
Then we have the Weitzenb\"ock formula
\begin{equation}\label{eq321}
\begin{split}
\sqrt{-1}\Lambda_{\og}\partial\bp H(s, s)
=&\mid \partial_H s \mid^2+H(s, -\sqrt{-1}\Lambda_{\og}\bp\pt_H s)\\
=&\mid \partial_H s \mid^2+H(s, -\mfk_H (s))+H(s, \sqrt{-1}\Lambda_{\og}[\phi, \phi^{*H}]s)\\
\geq &H(s, -\mfk_H(s)).
\end{split}
\end{equation}

Since $(E,\bp_E,\phi)$ admits an approximate Hermitian-Einstein structure, it holds that for $\forall \xi>0$, there exists a metric $H_{\xi}$ such that
\begin{equation*}
\sup\limits_X\mid \mathcal{K}_{H_{\xi}}-\lambda \mathrm{Id} \mid<\xi,
\end{equation*}
where $\lmd=\displaystyle{\frac{2\pi\deg(E)}{\mathrm{Vol}(X)\mathrm{rank(E)}}}<0.$
Taking $\xi=\frac{-\lmd}{2},$ there exists a Hermitian metric $H_{\frac{-\lmd}{2}}$ such that
\begin{equation}\label{eq322}
\frac{3\lmd}{2}\cdot \mathrm{Id}<\mfk_{H_{\frac{-\lmd}{2}}}<\frac{\lmd}{2}\cdot \mathrm{Id}.
\end{equation}
Combining (\ref{eq321}) with (\ref{eq322}), we have
\begin{equation}\label{eq323}
-\frac{\lmd}{2}\mid s \mid^2 \leq \sqrt{-1}\Lambda_{\og}\partial\bp H_{\frac{-\lmd}{2}}(s, s).
\end{equation}
Integrating both sides of (\ref{eq323}) over $X$ and using $\pt\bp \og^{n-1}=0$, there holds
\begin{equation*}
\begin{split}
0\leq \int_X \mid s \mid_{H_{\frac{-\lmd}{2}}}^2 \displaystyle{\frac{\og^n}{n!}}\leq \int_X \sqrt{-1}\Lambda_{\og}\partial\bp H_{\frac{-\lmd}{2}}(s, s)=-\int_X <s,s>_{H_{\frac{-\lmd}{2}}} \bp\partial \displaystyle{\frac{\og^{n-1}}{(n-1)!}}=0.
\end{split}
\end{equation*}
This forces $s=0.$
\end{proof}

\textbf{Proof of Theorem \ref{tm2}}
  Let $\mathcal{F}$ be any saturated Higgs sub-sheaf with rank $p$. Construct a Higgs bundle
$$\mathcal{G}=(G, \vartheta)=(\wedge^p E\otimes \det{\mathcal{F}}^{-1},\vartheta),$$
where $\vartheta$ is the induced Higgs field. By using the technique in \cite[p.~119]{Kobayashi1}, one can check that $\mathcal{G}$ admits an approximate Hermitian-Einstein structure with the constant
\begin{equation}\label{D.3}
\lmd(\mathcal{G})=\displaystyle{\frac{2p\pi}{\mathrm{Vol}(X)}}(\mu(E)-\mu(\mathcal{F})).
\end{equation}
  The canonical morphism $ \det{\mathcal{F}}\hookrightarrow  \wedge^p E$ induced by the inclusion map $i: \mf\hookrightarrow E$ can be seen as a non-trivial $\vartheta$-invariant section of $\mathcal{G}$. Then from Proposition \ref{prop2}, we have $\lmd(\mathcal{G})=\displaystyle{\frac{2\pi\deg(G)}{\mathrm{Vol}(X)\mathrm{rank(G)}}}\geq 0$. This together with (\ref{D.3}) indicates $\mu(\mf)\leq \mu(E),$ i.e. $(E,\bp_E,\phi)$ is  semi-stable.
\qed

\par
\bibliographystyle{plain}

\end{document}